\documentclass[12pt, a4paper]{amsart}

\usepackage[english]{babel}
\usepackage{amsmath}
\usepackage{amssymb}
\usepackage{bm}
\usepackage{amscd} 
\usepackage{microtype} 
\usepackage{verbatim} 
\usepackage{color}  
\usepackage{tikz}\usetikzlibrary{arrows.meta}
\usepackage{enumitem}
\usepackage{mathtools} 
\usepackage{cite}
\usepackage{hyperref} 
\usepackage[initials,msc-links]{amsrefs} 
\usepackage{dynkin-diagrams}

\usepackage[OT2, T1]{fontenc}
\title{Profinite recognition of uniformity}

\author[H. Kammeyer]{Holger Kammeyer}

 \address{Heinrich Heine University D{\"u}sseldorf, Faculty of Mathematics and Natural Sciences, Department of Mathematics, Germany}
 \email{holger.kammeyer@hhu.de}
 

\theoremstyle{plain}
\newtheorem{theorem}[equation]{Theorem}

\newtheorem{corollary}[equation]{Corollary}
\newtheorem{proposition}[equation]{Proposition}

\theoremstyle{definition}

\newtheorem*{definition*}{Definition}

\newtheorem*{observation*}{Observation}

\providecommand{\ignore}[1]{}

\providecommand{\R}{\mathbb{R}}
\providecommand{\Q}{\mathbb{Q}}
\providecommand{\Z}{\mathbb{Z}}

\providecommand{\C}{\mathbb{C}}

\DeclareMathOperator{\Sym}{Sym}

\newcommand*{\arXiv}[1]{ \href{http://www.arxiv.org/abs/#1}{arXiv:\textbf{#1}}}

\begin{document}

\begin{abstract}
  Let \(G\) be a semisimple Lie group that is not isogenous to the isometry group of real or complex hyperbolic \(n\)-space.  We show that the rational rank is a profinite invariant for irreducible lattices in \(G\), unless \(G\) has type \(D_{2n}\) which forms a genuine exception.  In particular, profinitely isomorphic lattices in \(G\) of type~\(D_{2n}\) can have differing virtual cohomological dimension.  As another corollary, uniformity is a profinite property for irreducible lattices in \(G\) outside type \(D_4\).  Assuming CSP, the existence of triality symmetries gives uniform and non-uniform profinitely isomorphic lattices in type \(D_4\).
\end{abstract}

\maketitle

\section{Introduction}

It is an active area of research to investigate which invariants and properties of a finitely generated, residually finite group \(\Gamma\) are \emph{profinite}, meaning that they can be inferred from the profinite completion~\(\widehat{\Gamma}\).  Lattices \(\Gamma \le G\) in semisimple Lie groups provide an enlightening test ground for these questions.  An important invariant of such lattices is the \emph{\(\Q\)-rank}.  If \(\Gamma\) is arithmetic in the sense that a finite central quotient is commensurable with an arithmetic subgroup of a simple algebraic \(k\)-group \(\mathbf{G}\) over a number field \(k\), then \(\operatorname{rank}_\Q (\Gamma) = \operatorname{rank}_k \mathbf{G}\).  There exist, however, also intrinsic geometric and cohomological characterizations of \(\operatorname{rank}_\Q (\Gamma)\) as explained in~\cite{Witte-Morris:arithmetic-groups}*{\S 2.2, p.\,23}.  It can for example be described as the dimension of the \emph{asymptotic cone} of the locally symmetric space \(\Gamma \backslash G / K\) where \(K \le G\) is a maximal compact subgroup.  Intuitively, the asymptotic cone is the simplicial complex arising when ``looking at \(\Gamma \backslash G / K\) from afar''.   As such, the invariant \(\operatorname{rank}_\Q (\Gamma)\) can be understood as a measure of non-uniformity.  In particular, \(\Gamma\) is uniform, meaning \(G/\Gamma\) is compact, if and only if \(\operatorname{rank}_\Q (\Gamma) = 0\).

\begin{theorem} \label{thm:same-q-rank}
  Let \(\Gamma, \Lambda \le G\) be irreducible lattices in a connected semisimple Lie group \(G\) with finite center and no compact factors.  Assume that \(G\) is neither isogenous to \(\operatorname{SO}^0(n,1)\) nor to \(\operatorname{SU}(n,1)\) for any \(n \ge 2\).  Suppose moreover that \(G\) does not have absolute type \(D_{2n}\) for any \(n \ge 2\).  If \(\Gamma\) and \(\Lambda\) have commensurable profinite completions, then
\[ \operatorname{rank}_\Q (\Gamma) = \operatorname{rank}_\Q (\Lambda). \]
\end{theorem}

The exclusion of the groups \(\operatorname{SO}^0(n,1)\) and \(\operatorname{SU}(n,1)\) from the theorem has the effect that the arithmeticity theorems of Margulis~\cite{Margulis:discrete-subgroups}, Gromov--Schoen~\cite{Gromov-Schoen:harmonic}, and Corlette~\cite{Corlette:superrigidity} apply.  Hence \(\Gamma\) and \(\Lambda\) are arithmetic in the above sense.  In particular, all simple factors of \(G\) have the same complex type which we call the \emph{absolute type} of \(G\).  The case of absolute type \(D_{2n}\) cannot be included in the theorem.

\begin{theorem} \label{thm:different-q-rank}
  For every \(n \ge 2\), there exists a connected semisimple Lie group \(G\) of absolute type \(D_{2n}\) with higher rank, finite center, and no compact factors, as well as irreducible lattices \(\Gamma, \Lambda \le G\) such that \(\Gamma\) and \(\Lambda\) have commensurable profinite completions but
  \[ \operatorname{rank}_\Q (\Gamma) = n, \qquad \operatorname{rank}_\Q (\Lambda) = n-1.\]
  If \(n \ge 3\), no other two distinct \(\Q\)-ranks occur for any two such lattices.
\end{theorem}

Our methods will more precisely reveal that for \(n \ge 3\), such lattices exist if and only if the simple factors of \(G\) are isogenous to either
\[ \operatorname{SO}^0(2n,2n), \ \operatorname{SO}^*(4n), \text{ or } \operatorname{SO}(4n; \C). \]
Note that Theorems~\ref{thm:same-q-rank} and~\ref{thm:different-q-rank} say in particular that uniformity is a profinite property for higher rank lattices outside type \(D_4\).  In type \(D_4\), the situation is more subtle.  We will demonstrate that the existence of triality symmetries should give counterexamples to the profiniteness of uniformity in type \(D_4\).  But for the moment, the result must remain conditional on \emph{Serre's conjecture} that every simply connected simple \(k\)-group with \(\sum_{v \mid \infty} \operatorname{rank}_{k_v} \mathbf{G} \ge 2\) should have the \emph{congruence subgroup property (CSP)}~\cite{Prasad-Rapinchuk:survey}.

\begin{theorem} \label{thm:triality}
  Suppose Serre's conjecture is true for anisotropic groups of type \({}^2 D_4\).  Then there exists a connected semisimple Lie group \(G\) of absolute type \(D_4\) and irreducible lattices \(\Gamma, \Lambda \le G\) such that \(\Gamma\) and \(\Lambda\) have commensurable profinite completions but
  \[ \operatorname{rank}_\Q (\Gamma) = 1, \qquad \operatorname{rank}_\Q (\Lambda) = 0.\]
\end{theorem}

More precisley, the lattice \(\Lambda\) in the theorem can be described as an arithmetic subgroup of the spinor group \(\mathbf{G} = \operatorname{Spin}(h)\) of a globally anisotropic skew-hermitian form \(h\) in four quaternionic variables.  Even though \(\mathbf{G}\) has type \({}^2 D_4\), hence is not a trilatarian form, A.\,Rapinchuk informed us that centrality of the congruence kernel is currently not known in this case (but should not be entirely out of reach either).

Because of the geometric relevance, we state the immediate corollaries of our results on the profiniteness of uniformity as separate theorems.

\begin{theorem} \label{thm:uniform-profinite}
  Let \(\Gamma, \Lambda \le G\) be irreducible lattices in a connected semisimple Lie group \(G\) with finite center and no compact factors.  Assume \(G\) is neither isogenous to \(\operatorname{SO}^0(n,1)\) nor to \(\operatorname{SU}(n,1)\) for any \(n \ge 2\).  Suppose moreover \(G\) is not of absolute type \(D_4\) and \(\Gamma\) and \(\Lambda\) have commensurable profinite completions.  Then \(\Gamma\) is uniform if and only if \(\Lambda\) is.
\end{theorem}

\begin{theorem} \label{thm:uniform-non-uniform} 
  Suppose Serre's conjecture is true for anisotropic groups of type \({}^2 D_4\).  Then there exists a connected semisimple Lie group \(G\) of absolute type~\(D_4\) and irreducible lattices \(\Gamma, \Lambda \le G\) with commensurable profinite completions such that \(\Gamma\) is non-uniform and \(\Lambda\) is uniform.
\end{theorem}

In the last theorem, various choices for \(G\) are possible, one of them being \(G = \operatorname{SO}^0(7,1) \times \operatorname{SO}^0(7,1)\).  In that case \(\Gamma\) and \(\Lambda\) act irreducibly on \(\mathbb{H}^7 \times \mathbb{H}^7\), the product of two copies of 7-dimensional hyperbolic space, and \(\Gamma\) has a non-compact quotient while \(\Lambda\) has a compact quotient.

\medskip
We briefly discuss the situation in the remaining case of lattices in \(\operatorname{SO}^0(n,1)\) and \(\operatorname{SU}(n,1)\) for \(n \ge 2\).  For these lattices, profiniteness of uniformity is closely related to the well-known problem whether such lattices are \emph{good} in the sense of Serre~\cite{Serre:galois-cohomology}*{p.\,16}.  Indeed, it is clear that the \emph{virtual cohomological dimension} (vcd) of a good group is captured in the profinite completion.  For real rank one Lie groups, the vcd of a non-uniform lattice is one less than the vcd of a uniform lattice (which is \(n\) for \(\operatorname{SO}^0(n,1)\) and \(2n\) for \(\operatorname{SU}(n,1)\)).  Hence any two good lattices in \(\operatorname{SO}^0(n,1)\) or \(\operatorname{SU}(n,1)\) are either both uniform or both non-uniform if they have commensurable profinite completions.

All lattices in \(\operatorname{SO}^0(3,1)\) are good as a consequence of virtual fibering~\cites{Agol:haken, Wise:quasiconvex-hierarchy}.  Moreover, \emph{arithmetic hyperbolic lattices of simplest type} are good~\cite{Kolpakov-et-al:embedding}*{Theorem~6.5}.  This includes all non-uniform arithmetic lattices in all \(\operatorname{SO}^0(n,1)\) and also all uniform arithmetic lattices if \(n\) is even.  Establishing goodness is particularly intriguing for lattices in \(\operatorname{SU}(n,1)\), where for \(n \ge 2\) not a single lattice is known to be good or non-good.  Some recent efforts were undertaken by M.\,Stover in~\cite{Stover:ball-quotients} who shows in particular that a cocompact arithmetic lattice \(\Gamma \le \operatorname{SU}(n,1)\) of simplest type satisfies \(\operatorname{vcd}(\widehat{\Gamma}) \ge 2n\).  We remark that even more recently, Xiaoyu Xu~\cite{Xu:psl2c} obtained the seminal result that lattices in \(\operatorname{SO}^0(3,1) \cong \operatorname{PSL_2}(\C)\) are profinitely rigid among themselves, hence profinitely isomorphic lattices in \(\operatorname{SO}^0(3,1)\) are even conjugate by Mostow rigidity.

\medskip
Since Theorem~\ref{thm:uniform-profinite} says that for the most part, uniformity is a profinite property for lattices in the same Lie group, one may ask to what extent uniformity remains profinite if the lattices are allowed to lie in different Lie groups.  Many other properties (Property \((T)\)~\cite{Aka:property-t}, higher \(\ell^2\)-Betti numbers~\cite{Kammeyer-Sauer:spinor}, Euler characteristic~\cite{Kammeyer-et-al:profinite-invariants}, bounded cohomology~\cite{Echtler-Kammeyer:bounded-cohomology}) fail to be profinite for lattices in different Lie groups and uniformity is no exception.  The failure is most striking in certain exceptional types.

\begin{theorem} \label{thm:different-lie-groups}
  Let \(G\) be a connected non-compact real or complex simple Lie group with finite center of type \(E_8\), \(F_4\), or \(G_2\) but not of real type \(F_{4(-20)}\).  Then for every uniform lattice \(\Gamma \le G\), there exists a semisimple Lie group \(H\) and a non-uniform irreducible lattice \(\Lambda \le H\) such that \(\Gamma\) and \(\Lambda\) have commensurable profinite completions.
\end{theorem}

Theorem~\ref{thm:different-lie-groups} does not immediately carry over to lattices in simple Lie groups of other Cartan--Killing types.  In particular, there exist \emph{congruence rigid} algebraic groups \cite{Baumann-Kammeyer:congruence-rigidity}.  The arithmetic subgroups of some of these define uniform lattices in higher rank semisimple Lie groups.  For example any central division algebra \(D\) of degree three over \(\Q\) determines the anisotropic \(\Q\)-group \(\mathbf{G} = \operatorname{SL_1}(D)\) of type \({}^1 A_2\) and an arithmetic subgroup \(\Gamma \le \mathbf{G}\) is a uniform lattice in \(\operatorname{SL}_3(\R)\).  According to Serre's conjecture, though again open in this particular case, \(\mathbf{G}\) should have CSP.  Assuming \(\mathbf{G}\) does have CSP, another open problem asks whether the arithmetic subgroups of an algebraic group with CSP possess proper Grothendieck subgroups~\cite{Platonov-Rapinchuk:algebraic-groups}*{Problem p.\,434}.  If the answer is negative, then \(\Gamma\) is \emph{profinitely solitary} so that the commensurability class of \(\Gamma\) is determined by the commensurability class of \(\widehat{\Gamma}\) among all finitely generated residually finite groups~\cite{Baumann-Kammeyer:congruence-rigidity}*{Theorem~4}.  By strong Mostow rigidity~\cite{Margulis:discrete-subgroups}*{Theorem VII.7.5, p.\,254}, Theorem~\ref{thm:uniform-profinite} implies that uniformity would be a profinite property for \(\Gamma\) among all lattices in all semisimple Lie groups.

\medskip
The assumptions on \(\Gamma \le G\) in Theorems~\ref{thm:same-q-rank} and~\ref{thm:different-q-rank} imply that \(\Gamma\) is an arithmetic lattice.  Assuming that \(G\) has trivial center, the group \(\Gamma\) is moreover virtually torsion-free.  Thus an application of the Borel--Serre compactification gives that the virtual cohomological dimension is
\[ \operatorname{vcd}(\Gamma) = \dim G -\dim K - \operatorname{rank}_\Q (\Gamma) \]
where \(K \le G\) is a maximal compact subgroup~\cite{Borel-Serre:corners}*{Corollary~11.4.3}.  Hence we obtain the following consequence on the virtual cohomological dimension.

\begin{corollary}
\begin{enumerate}[label=(\roman*), leftmargin=2em, labelsep=0.6em, itemsep=0.6em, topsep=0.4em]
\item Let \(\Gamma, \Lambda \le G\) be irreducible lattices in a connected semisimple Lie group \(G\) with trivial center and no compact factors.  Assume that \(G\) is neither isogenous to \(\operatorname{SO}^0(n,1)\), nor to \(\operatorname{SU}(n,1)\), and does not have type \(D_{2n}\) for any \(n \ge 2\). If \(\Gamma\) and \(\Lambda\) have commensurable profinite completions, then
\[
  \operatorname{vcd}(\Gamma) = \operatorname{vcd}(\Lambda).
\]

\item For every \(n \ge 2\), there exist irreducible lattices \(\Gamma, \Lambda \le G\) in certain semisimple Lie groups of absolute type \(D_{2n}\) such that \(\Gamma\) and \(\Lambda\) have commensurable profinite completions but
  \[ \operatorname{vcd}(\Gamma)=4n^2-n, \qquad \operatorname{vcd}(\Lambda)=4n^2-n+1.
\]
\end{enumerate}
\end{corollary}

\medskip
The proof strategy for Theorem~\ref{thm:same-q-rank} can be subsumed as follows.  Arithmeticity and adelic superrigidity~\cite{Kammeyer-Kionke:adelic-superrigidity} apply, hence we can assume that the involved lattices are given by arithmetic subgroups of absolutely almost simple algebraic groups defined over number fields and these groups are isomorphic over an isomorphism of the finite adele rings.  In particular, the same set of local Tits indices arises for both groups.  To prove the theorem, one would need to infer the global rank from the local Tits indices.  It will turn out that this is possible whenever the set of distinguished nodes in all Tits indices is invariant under the Dynkin diagram symmetries.  Only in type \(D_{2n}\), this is not granted and can indeed be exploited to prove Theorem~\ref{thm:different-q-rank}.  The classification of Tits indices implies that only the given values of \(\Q\)-ranks can occur if \(n \ge 3\).  In type \(D_4\) however, the additional triality symmetries of the Dynkin diagram lead to Theorem~\ref{thm:triality}.  To prove Theorem~\ref{thm:different-lie-groups}, we can draw on the fact that in these exceptional types, there exists only the split form over \(\mathfrak{p}\)-adic fields so that the congruence completion of the given uniform arithmetic group is commensurable with the congruence completion of an arithmetic subgroup of the global split form defined over the same number field.

\medskip
\noindent \textbf{Disclosure of AI assistance.}  The initial idea, results, and proofs were developed without AI assistance.  My first proof of Proposition~\ref{prop:distinguished-orbits}, however, was extracted from the methods of Prasad--Rapinchuk in~\cite{Prasad-Rapinchuk:prescribed}.  As an alternative, OpenAI ChatGPT 5.6 Sol suggested subsequently the more streamlined argument to apply the Harder local-global principle for homogeneous projective varieties to the varieties of \(k\)-parabolic subgroups.  For the convenience of the reader, I incorporated this shorter proof after a careful check and provided the relevant references.  ChatGPT was additionally applied to search the literature, to create the Tits diagram pictures, and for minor corrections at the final stage. In contrast, it was not used for structuring or wording the article.  I wrote the text entirely by myself and I remain responsible for its correctness.

\medskip
\noindent \textbf{Acknowledgements.} I am grateful to Giles Gardam and Ursula Hamenst{\"a}dt who asked about the profiniteness of uniformity when I was giving a talk in the Bonn GGT seminar earlier this year.  This work was supported by the GRK ``Algebro-Geometric Methods in Algebra, Arithmetic and Topology'', DFG 284078965.  I wish to thank Adrian Baumann, Daniel Echtler, and Andrei Rapinchuk for helpful discussions.

\section{Tits indices and local-global principles}

In this section, we prove the local-global principles around Tits indices of simple algebraic groups over number fields that lie at the core of our results.  Tits indices were introduced and classified in~\cite{Tits:classification}.  We briefly recall the definition.

Let \(K\) be a field, for our purposes of characteristic zero, and let \(\mathbf{G}\) be a connected and simply connected absolutely almost simple linear algebraic group defined over \(K\) (a \emph{K-group} in what follows).  Fix an algebraic closure \(\overline{K} / K\) and denote by \(\operatorname{Gal}(K) = \operatorname{Gal}(\overline{K}/K)\) the absolute Galois group of \(K\).  We choose a maximal \(K\)-split \(K\)-torus \(\mathbf{S} \le \mathbf{G}\) and a maximal \(K\)-torus \(\mathbf{T} \le \mathbf{G}\) with \(\mathbf{S} \le \mathbf{T}\).  Choose compatible orders of the character lattices \(X^*(\mathbf{S})\) and \(X^*(\mathbf{T})\) and let \(\Delta\) be a corresponding system of simple roots with respect to \(\mathbf{T}\).  We denote by \(\Delta_0 \subseteq \Delta\) the subset of those simple roots that vanish on \(\mathbf{S}\).  The elements of \(\Delta\) can be displayed in a suitable order as the nodes of the Dynkin diagram of \(\mathbf{G}\).  The Galois group \(\operatorname{Gal}(K)\) acts on the Dynkin diagram \(\Delta\) by the so-called \emph{\(\ast\)-action}: We have a \(\operatorname{Gal}(K)\)-action on \(X^*(\mathbf{T})\) such that every \(\sigma \in \operatorname{Gal}(K)\) translates \(\Delta\) to a new system of simple roots \(\sigma.\Delta\).  There exists a unique element \(w\) of the Weyl group with \(w.\sigma.\Delta = \Delta\) and we set \(\sigma^* = w \circ \sigma\).  The group \(\operatorname{Gal}(K)\) acts by diagram symmetries and \(\Delta_0\) is an invariant subset.

The \emph{Tits index} of \(\mathbf{G}\) consists of the Dynkin diagram \(\Delta\), the \(\operatorname{Gal}(K)\)-\(\ast\)-action on \(\Delta\), and the invariant subset \(\Delta_0\).  We obtain an associated \emph{Tits diagram} by drawing nodes in the Dynkin diagram close to one another whenever they lie in the same \(\operatorname{Gal}(K)\)-orbit and circling the orbits in the complement of \(\Delta_0\).  The latter orbits are called \emph{distinguished}.  The terminology of a Tits diagram is not standard.  We introduce the notion to describe what is left from a Tits index after discarding the \(\operatorname{Gal}(K)\)-action (and in particular the field \(K\)) but retaining the orbits and the information which of these are distinguished.  It is clear what it means for two Tits diagrams to be isomorphic.

For what follows, it will be convenient to note that Tits indices and diagrams can also be defined without reference to a \(K\)-split torus \(\mathbf{S}\), using instead the system of parabolic subgroups in \(\mathbf{G}\).  Indeed, given any maximal \(K\)-torus \(\mathbf{T} \le \mathbf{G}\) and system of simple roots \(\Delta\), the simple roots are in one-to-one correspondence with the \(\overline{K}\)-conjugacy classes of maximal parabolic subgroups of \(\mathbf{G}\).  These conjugacy classes form homogeneous varieties defined over~\(\overline{K}\), hence they are permuted by \(\operatorname{Gal}(K)\), defining the \(\operatorname{Gal}(K)\)-\(\ast\)-action on \(\Delta\) from above.  Since all maximal \(K\)-tori of \(\mathbf{G}\) are conjugate, one of them contains a maximal \(K\)-split torus.  We thus obtain from~\cite{Borel-Tits:reductifs}*{(1) in Section~6.3 and (2) in Section~6.4} that an orbit \(\Omega \subset \Delta\) of the \(\ast\)-action is distinguished if and only if the conjugacy class of the corresponding standard parabolic subgroup \(\mathbf{P}_{\Delta \setminus \Omega}\) has a representative defined over \(K\).  In that way, we recover the full information stored in the Tits diagram of \(\mathbf{G}\).  From this viewpoint, we may use the pair \((\mathbf{T}, \Delta)\) chosen in \(\mathbf{G}\) also for every base change \(\mathbf{G} \times_K L\) over a field extension \(L/K\) so that we can track the effect of base change on the Tits diagram even though it is not clear whether \(\mathbf{T} \times_K L\) contains a maximal \(L\)-split torus of \(\mathbf{G} \times_K L\).

\smallskip
From now on, we consider the special case that \(K = k\) is a number field.  When base changing \(\mathbf{G}\) from \(k\) to any completion \(k_v\) at a finite or infinite place, the Tits index may change in two possible ways: The set \(\Delta_0\) can shrink and the effectively acting quotient group of the \mbox{\(\ast\)-action} may be smaller if it was not trivial anyway (outside type \(D_4\), this means the action becomes trivial).  The next two propositions give some insight on the question how much of the global Tits diagram of \(\mathbf{G}\) can be recovered from the collection of local Tits diagrams.

\begin{proposition} \label{prop:distinguished-orbits}
  Let \(\mathbf{G}\) be a \(k\)-group and let \(\Omega_1, \ldots, \Omega_r\) be the orbits in the Tits diagram of \(\mathbf{G}\).  Then \(\Omega_i\) is distinguished if and only if it consists of distinguished orbits over every place \(v\) of \(k\).
\end{proposition}

\begin{proof}
    Given a subset \(I \subset \Delta\) that is invariant under the \(\operatorname{Gal}(k)\)-\(\ast\)-action on \(\Delta\), we let \(\mathbf{X}_I\) be the projective homogeneous variety of parabolic subgroups of \(\mathbf{G}\) that are conjugate to the standard parabolic subgroup of \(\mathbf{G}\) obtained by omitting \(I\) from \(\Delta\).  Since \(I\), hence also \(\Delta \setminus I\), is \(\ast\)-invariant, we have that \(\mathbf{X}_I\) is defined over \(k\) as explained in~\cite{Borel-Tits:reductifs}*{(I) in Section~6.4}.  Moreover, for every field extension \(L/k\), the variety \(\mathbf{X}_I\) has an \(L\)-rational point or, equivalently, one of the parabolic subgroups can be defined over \(L\), if and only if \(I\) consists of distinguished orbits in the Tits diagram of \(\mathbf{G}\) over \(L\)~\cite{Borel-Tits:reductifs}*{(I) in Section~6.3}.

  By the Harder local-global theorem for projective homogeneous varieties~\cite{Harder:bericht}*{Satz 4.3.3}, with the missing \(E_8\)-case later provided by Chernousov~\cite{Chernousov:hasse}, the variety \(\mathbf{X}_I\) has a \(k\)-rational point if and only if it has a \(k_v\)-rational point for every place \(v\) of \(k\).  Applying this principle to \(I = \Omega_i\) completes the proof.
\end{proof}

\begin{proposition} \label{prop:isomorphic-tits-indices}
Let \(\mathbf{G}\) be a \(k\)-group and let \(\mathbf{H}\) be an \(l\)-group.  Assume that there exist local isomorphisms \(\eta_v \colon l_v \rightarrow k_{v'}\) over a bijection \(v \mapsto v'\) of the places of \(l\) and \(k\), and corresponding isomorphisms \(\mathbf{G} \cong_{\eta_v} \mathbf{H}\).  Suppose that \(\mathbf{G}\) and \(\mathbf{H}\) are not of type \(D_{2n}\).  Then \(\mathbf{G}\) and \(\mathbf{H}\) have isomorphic Tits diagrams.
\end{proposition}

\begin{proof}
  Let \(\mathbf{G_0}\) and \(\mathbf{H_0}\) be the unique quasi-split inner forms of \(\mathbf{G}\) and \(\mathbf{H}\), respectively.  Let \(K\) be the splitting field of \(\mathbf{G_0}\), in other words the Galois extension \(K/k\) defined by the kernel of the \(\ast\)-action of \(\operatorname{Gal}(k)\) on the Dynkin diagram of \(\mathbf{G_0}\).  Since by assumption, \(\mathbf{G_0}\) does not have type \(D_4\), we have that \([K : k] = 1 \text{\,or } 2\).  In the latter case, \(\mathbf{G_0}\) is non-split and remains non-split at every place \(v\) of \(k\) that is inert in \(K\).

  By uniqueness of quasi-split forms, the family of isomorphisms \(\mathbf{G} \cong_{\eta_v} \mathbf{H}\) provides a family of isomorphisms \(\mathbf{G_0} \cong_{\eta_v} \mathbf{H_0}\).  We conclude that \(\mathbf{G_0}\) is \(k\)-split if and only if \(\mathbf{H_0}\) is \(l\)-split.  Note that since type \(D_4\) is excluded, there are at most two different Tits diagrams of quasi-split groups: the split one, where each node forms a distinguished orbit, and in type \(A_{n \ge 2}\), \(D_{n \ge 5}\), and \(E_6\) an additional one where all orbits of the unique non-trivial symmetry are distinguished.  We conclude that \(\mathbf{G_0}\) and \(\mathbf{H_0}\) have isomorphic global Tits diagrams.

  The classification in~\cite{Tits:classification} reveals that over real and \(\mathfrak{p}\)-adic fields and outside type \(D_{2n}\), all possible Tits diagrams are invariant under the symmetries of the underlying Dynkin diagram.  Therefore, for each place \(v\) of \(l\), the isomorphism \(\mathbf{G} \cong_{\eta_v} \mathbf{H}\) shows that the same well-defined set of orbits is distinguished in the Tits diagrams of \(\mathbf{G}\) over \(k_{v'}\) and \(\mathbf{H}\) over \(l_v\).  In view of Proposition~\ref{prop:distinguished-orbits}, the proof is complete.  
\end{proof}

\begin{corollary} \label{cor:rank}
  Under the assumptions of Proposition~\ref{prop:isomorphic-tits-indices}, we have that
  \[ \operatorname{rank}_k \mathbf{G} = \operatorname{rank}_l \mathbf{H}. \]
\end{corollary}

\begin{proof}
The rank equals the number of distinguished orbits in the Tits diagram.
\end{proof}

It turns out that the existence of asymmetric Tits diagrams in type \(D_{2n}\) precludes its inclusion in the corollary.

\begin{proposition} \label{prop:differing-ranks}
  For every number field \(k\) and every \(n \ge 2\), there exist \(k\)-groups \(\mathbf{G}\) and \(\mathbf{H}\) of type \(D_{2n}\) with \(\mathbf{G} \cong_{k_v} \mathbf{H}\) for all places \(v\) of \(k\) but
  \[ \operatorname{rank}_k \mathbf{G} = n, \qquad \operatorname{rank}_k \mathbf{H} = n - 1.\]
\end{proposition}

The proof needs some preparation.  To begin, we recall some standard concepts from Galois cohomology as outlined in~\cite{Serre:galois-cohomology}*{Chapter~I,~\S5 and Chapter~III}.  Let \(K\) be again a field of characteristic zero and fix an algebraic closure \(\overline{K}\).  For a \(\operatorname{Gal}(K)\)-group \(A\), the first non-commutative Galois cohomology set with values in \(A\) is denoted by \(H^1(K, A)\).  If \(A = \operatorname{Aut} \mathbf{G}\) for a \(K\)-group \(\mathbf{G}\) and \([a] \in H^1(K, \operatorname{Aut} \mathbf{G})\), then \({}_a \mathbf{G}\) denotes the \(K\)-form of \(\mathbf{G}\) obtained from \(\mathbf{G}\) by twisting with the 1-cocycle \(a\).

For what follows, it is important that we specify precisely what we mean when we refer to an \emph{inner \(K\)-form} of \(\mathbf{G}\).  An inner \(K\)-form of \(\mathbf{G}\) is often just defined as a \(K\)-group \(\mathbf{H}\) that is \(K\)-isomorphic to \({}_a \mathbf{G}\) for some 1-cocycle \(a\) with values in \((\operatorname{Ad} \mathbf{G})(\overline{K})\).  Here \(\operatorname{Ad} \mathbf{G}\) is the \emph{adjoint form} of \(\mathbf{G}\), the subgroup of inner automorphisms in \(\operatorname{Aut} \mathbf{G}\).  With this definition, \(K\)-isomorphism classes of inner \(K\)-forms are classified by the image of \(H^1(K, \operatorname{Ad} \mathbf{G})\) in \(H^1(K, \operatorname{Aut} \mathbf{G})\).  Instead, we adopt the definition of Milne~\cite{Milne:reductive-groups}*{6.35} according to which an \emph{inner \(K\)-form} of \(\mathbf{G}\) is a pair \((\mathbf{H}, f)\) consisting of a \(K\)-group \(\mathbf{H}\) and a \(\overline{K}\)-isomorphism \(f \colon \mathbf{G} \times_K \overline{K} \rightarrow \mathbf{H} \times_K \overline{K}\) such that the corresponding 1-cocycle \(a_\sigma = f^{-1} \circ {}^\sigma f\) has values in \((\operatorname{Ad} \mathbf{G})(\overline{K})\) for all \(\sigma \in \operatorname{Gal}(K)\).  Two inner \(K\)-forms \((\mathbf{H_1}, f_1)\) and \((\mathbf{H_2}, f_2)\) are called \emph{equivalent} if there exists a \(K\)-isomorphism \(\varphi \colon \mathbf{H_1} \rightarrow \mathbf{H_2}\) such that \(f_2^{-1} \circ \varphi \circ f_1 \in (\operatorname{Ad} \mathbf{G})(\overline{K})\).  With this definition, equivalence classes of inner \(K\)-forms are classified by the Galois cohomology set \(H^1(K, \operatorname{Ad} \mathbf{G})\) itself.

Assume now that \(\mathbf{G}\) is \(K\)-quasi-split, pick a maximal \(K\)-torus \(\mathbf{T}\) and let \(\Delta\) be a system of simple roots with respect to \(\mathbf{T}\).  These choices yield a short exact sequence
\begin{equation} \label{eq:ses-ad-aut-sym} 1 \rightarrow \operatorname{Ad} \mathbf{G} \longrightarrow \operatorname{Aut} \mathbf{G} \longrightarrow \operatorname{Sym} \Delta \rightarrow 1 \end{equation}
where \(\operatorname{Sym} \Delta\) is the \(K\)-group scheme of Dynkin diagram symmetries.  It is a constant group scheme unless \(\mathbf{G}\) is a non-split form of type \(D_4\).  The second map has the description that a \(\overline{K}\)-automorphism of \(\mathbf{G}\) permutes the \(\overline{K}\)-conjugacy classes of maximal parabolic subgroups of \(\mathbf{G}\).

Given an inner \(K\)-form \((\mathbf{H}, f)\) of \(\mathbf{G}\), we say that a \(\operatorname{Gal}(K)\)-\(\ast\)-orbit \(\Omega \subset \Delta\) is \emph{distinguished in \(\mathbf{H}\)} if the \(\overline{K}\)-conjugacy class of \(f(\mathbf{P}_{\Delta \setminus \Omega})\) in \(\mathbf{H}\) contains a representative defined over \(K\).  By construction, this property is well-defined on equivalence classes of inner forms.  Note that since \(\mathbf{G}\) is \(K\)-quasi-split, every orbit \(\Omega \subset \Delta\) is distinguished in \((\mathbf{G}, \operatorname{id}_{\mathbf{G} \times_K \overline{K}})\) itself.  Hence given \([a] \in H^1(K, \operatorname{Ad} \mathbf{G})\), there exists a well-defined way to drop circles from the \(\ast\)-orbits of the Tits diagram of \(\mathbf{G}\) to obtain the \emph{Tits diagram of \({}_a \mathbf{G}\) based on the Tits diagram of \(\mathbf{G}\)}.

\begin{proposition} \label{prop:swapping-tits}
Let \(\mathbf{G}\) be a \(K\)-quasi-split \(K\)-group.  Assume there exist two different classes \([a^1], [a^2] \in H^1(K, \operatorname{Ad} \mathbf{G})\) which become equal in \(H^1(K, \operatorname{Aut} \mathbf{G})\).  Then the Tits diagrams of \({}_{a^1} \mathbf{G}\) and \({}_{a^2} \mathbf{G}\) based on the Tits diagram of \(\mathbf{G}\) match after applying a nontrivial symmetry.
\end{proposition}

\begin{proof}
  The classes \([a^1]\) and \([a^2]\) correspond to equivalence classes of inner forms \(({}_{a^1} \mathbf{G}, f_1)\) and \(({}_{a^2} \mathbf{G}, f_2)\), respectively.  Since \([a^1]\) and \([a^2]\) become equal in \(H^1(K, \operatorname{Aut} \mathbf{G})\), there exists a \(K\)-isomorphism \(\varphi \colon {}_{a^1} \mathbf{G} \rightarrow {}_{a^2} \mathbf{G}\).

  The composition \(f_2^{-1} \circ \varphi \circ f_1\) is a \(\overline{K}\)-automorphism of \(\mathbf{G}\).  We claim that the corresponding Dynkin diagram symmetry \(s\) is non-trivial and sends roots distinguished in \({}_{a^1} \mathbf{G}\) to roots distinguished in \({}_{a^2} \mathbf{G}\).

Indeed, since \([a^1]\) and \([a^2]\) are different classes in \(H^1(K, \operatorname{Ad} \mathbf{G})\), the automorphism \(f_2^{-1} \circ \varphi \circ f_1\) is not inner, hence \(s\) is non-trivial by the short exact sequence~\eqref{eq:ses-ad-aut-sym}.  Now suppose that \(\Omega_1 \subset \Delta\) is distinguished in \({}_{a^1} \mathbf{G}\).  Then \(f_1(\mathbf{P}_{\Delta \setminus \Omega_1})\) is \(\overline{K}\)-conjugate to a parabolic subgroup \(\mathbf{P} \le {}_{a^1} \mathbf{G}\) defined over \(K\).  Since \(\varphi\) is a \(K\)-isomorphism, we have that \(\varphi(\mathbf{P}) \le {}_{a^2} \mathbf{G}\) is defined over \(K\) as well.  Thus \(f_2^{-1}(\varphi(\mathbf{P}))\) corresponds to an orbit \(\Omega_2 \subset \Delta\) distinguished in \({}_{a^2} \mathbf{G}\).  Moreover \(s(\Omega_1) = \Omega_2\) because the correspondence depends only on \(\overline{K}\)-conjugacy classes of parabolics.  Clearly, we obtain the inverse of \(s\) by interchanging the roles of \([a^1]\) and \([a^2]\) so \(s\) defines a bijection of the sets of distinguished orbits.
\end{proof}

We remark that all Tits diagrams of outer \(K\)-groups are invariant under all diagram symmetries commuting with the \(\ast\)-action so that Proposition~\ref{prop:swapping-tits} is only interesting if \(\mathbf{G}\) is \(K\)-split.  But the proof does not become any easier under this assumption, so we allow \(\mathbf{G}\) more generally to be \(K\)-quasi-split.

If in the last proposition, \(K\) is a \(\mathfrak{p}\)-adic field, then it follows from the description of the \((\operatorname{Sym} \Delta)(K)\)-action on \(H^1(K, \operatorname{Ad} \mathbf{G})\) given in~\cite{Kneser:galois}*{p.\,254} that such different classes \([a^1], [a^2] \in H^1(K, \operatorname{Ad} \mathbf{G})\) exist if and only if \(\mathbf{G}\) has type \({}^1 A_{n \ge 2}\), \({}^1 D_{n \ge 4}\), or \({}^1 E_6\).  If \(K = \R\), however, such classes only exist in type \({}^1 D_{2n}\) for \(n \ge 2\) as can be inferred from \cite{Adams-Taibi:galois}*{Section~10.3}.  More precisely, for \(n \ge 3\), there are two different inner forms of the real group \(\operatorname{Spin}(2n,2n)\), both isomorphic to \(\operatorname{Spin}^*(4n)\).  According to Proposition~\ref{prop:swapping-tits}, the Tits diagrams of these two inner forms based on the Tits diagram of \(\mathbf{G} = \operatorname{Spin}(2n,2n)\) are given by
\begin{equation} \label{eq:asymmetric-tits-diagrams}
\begin{tikzpicture}[scale=.8, transform shape, baseline=-0.5ex, line width=0.5pt]

  \newcommand{\sofourstarindex}{
    \coordinate (a1) at (0,0);
    \coordinate (a2) at (0.9,0);
    \coordinate (a3) at (1.8,0);

    \coordinate (c1) at (2.1,0);
    \coordinate (c2) at (2.75,0);

    \coordinate (b0) at (3.05,0);   
    \coordinate (b1) at (3.95,0);
    \coordinate (b2) at (4.85,0);
    \coordinate (b3) at (5.65,0.5);
    \coordinate (b4) at (5.65,-0.5);

    \draw (a1) -- (a2) -- (a3) -- (c1);
    \draw (c2) -- (b0) -- (b1) -- (b2);
    \draw (b2) -- (b3);
    \draw (b2) -- (b4);

    \node at (2.42,0) {$\cdots$};

    \foreach \v in {a1,a2,a3,b0,b1,b2,b3,b4}
      \fill (\v) circle[radius=1.4pt];

    \draw (a2) circle[radius=4.2pt];
    \draw (b0) circle[radius=4.2pt];
    \draw (b2) circle[radius=4.2pt];
    \draw (b4) circle[radius=4.2pt];
  }

  \begin{scope}
    \sofourstarindex
  \end{scope}
  
  \node at (6.75,0) {\large and};

  \begin{scope}[shift={(7.85,0)}, yscale=-1]
    \sofourstarindex
  \end{scope}
\end{tikzpicture}
\end{equation}
For \(n = 2\), there is another such form where precisely the left and the center node of the type \(D_4\) diagram are distinguished, explaining the accidental isomorphism \(\operatorname{Spin}(6,2) \cong \operatorname{Spin}^*(8)\).  Forms with these diagrams also exist over \(\mathfrak{p}\)-adic fields but these are the only Tits diagrams throughout all inner types that are not invariant under all diagram symmetries.  This special feature of type \({}^1 D_{2n}\) entails the contrasting behavior described in Proposition~\ref{prop:differing-ranks} as we now explain.

\begin{proof}[Proof of Proposition~\ref{prop:differing-ranks}.]
  Let \(\mathbf{G_0} = \operatorname{Spin}(2n,2n)\) be the unique \(k\)-split \(k\)-form of type \(D_{2n}\) and pick any four finite places \(v_1, v_2, v_3, v_4\) of \(k\).  We claim that there exists a class \(\alpha \in H^1(k, \operatorname{Ad} \mathbf{G_0})\) that localizes to a class with the first of the above pictured Tits diagrams based on the Tits diagram of \(\mathbf{G_0}\) at \(v_1\), \(v_2\), \(v_3\), and \(v_4\) and splits at all other places.  Similarly, there exists another class \(\beta \in H^1(k, \operatorname{Ad} \mathbf{G_0})\) that localizes to a class with the first diagram at \(v_1\) and \(v_2\), to a class with the second diagram at \(v_3\) and \(v_4\), and splits at all other places of \(k\).  Indeed, a necessary and sufficient condition for the realization of a family of local inner forms by a global inner form can be formulated by means of Poitou--Tate duality.  This is well-known to the experts and can for instance be inferred from~\cite{Prasad-Rapinchuk:prescribed}.  For a precise statement of the criterion we apply here, we refer however to \cite{Kammeyer:galois-cohomology}*{Theorem~IV.7.2, p.\,36} together with the description of the local Poitou--Tate maps in type \({}^1 D_{2n}\) given in~\cite{Baumann-Kammeyer:congruence-rigidity}*{Table~2, p.\,17}.  Under the isomorphism
  \[ H^1(k_{v_i}, \operatorname{Ad} \mathbf{G_0}) \xrightarrow{ \ \delta^1 \ } H^2(k_{v_i}, Z(\mathbf{G_0})) \cong \Z/2 \times \Z/2, \]
  the inner forms with the two Tits diagrams from above correspond to \((1,0)\) and \((0,1)\), respectively, and the split class corresponds to \((0,0)\).  The necessary and sufficient condition for realization by a global class is that all local classes sum up to \((0,0)\) which is granted in our construction.

    Let \(\mathbf{G}\) and \(\mathbf{H}\) be the \(k\)-groups corresponding to the images of \(\alpha\) and \(\beta\) in \(H^1(k, \operatorname{Aut} \mathbf{G_0})\), respectively.  By Proposition~\ref{prop:distinguished-orbits}, the Tits diagram of \(\mathbf{G}\) has \(n\) distinguished nodes while the Tits diagram of \(\mathbf{H}\) has \(n-1\) distinguished nodes.
\end{proof}

Next, we verify that \(n\) and \(n-1\) are the only distinct values that can occur as global rank for two locally isomorphic groups of type \(D_{2n}\) for \(n \ge 3\).

\begin{proposition} \label{prop:only-n-and-n-1}
Let \(\mathbf{G}\) be a \(k\)-group and let \(\mathbf{H}\) be an \(l\)-group, both of type \(D_{2n}\) for \(n \ge 3\).  Assume that there exist local isomorphisms \(\eta_v \colon l_v \rightarrow k_{v'}\) over a bijection \(v \mapsto v'\) of the places of \(l\) and \(k\), and corresponding isomorphisms \(\mathbf{G} \cong_{\eta_v} \mathbf{H}\).  If \(\operatorname{rank}_k \mathbf{G} > \operatorname{rank}_l \mathbf{H}\), then \(\operatorname{rank}_k \mathbf{G} = n\) and \(\operatorname {rank}_l \mathbf{H} = n-1\).
\end{proposition}

\begin{proof}
  Since \(n \ge 3\), we have that \(\Sym \Delta \cong \Z/2\).  It is thus clear from Proposition~\ref{prop:distinguished-orbits} that the two Tits diagrams in~\eqref{eq:asymmetric-tits-diagrams} occur for \(\mathbf{H}\) at two places of \(l\).  In contrast, only one of these two diagrams may arise for \(\mathbf{G}\) at any place of \(k\) and the corresponding distinguished end node \(\alpha\) is distinguished throughout for \(\mathbf{G}\).  Thus \(\operatorname{rank}_k \mathbf{G} \le n\) and \(\operatorname{rank}_l \mathbf{H} \le n-1\).  But in fact, these inequalities can never be strict because then one of the distinguished nodes on the long arm of the diagrams in~\eqref{eq:asymmetric-tits-diagrams} would have to be undistinguished at another place for both groups.  But the classification of Tits diagrams reveals that then also all nodes on the right of that node would be undistinguished at that place, contradicting that \(\alpha\) is distinguished everywhere.
\end{proof}

Note that Proposition~\ref{prop:differing-ranks} says in particular that for every \emph{positive} integer \(n\), the property of having global rank \(n\) is in general not determined by the local isomorphism types of a \(k\)-group.  In fact, neither is the property of being \(k\)-anisotropic in view of the following phenomenon occurring in type \(D_4\).

\begin{proposition} \label{prop:triality}
  Let \(k\) be a number field that is not totally imaginary.  Then there exist \(k\)-groups \(\mathbf{G}\) and \(\mathbf{H}\) of type \({}^2 D_4\) with \(\mathbf{G} \cong_{k_v} \mathbf{H}\) for all places \(v\) of \(k\) but
  \[ \operatorname{rank}_k \mathbf{G} = 1, \qquad \operatorname{rank}_k \mathbf{H} = 0. \]
\end{proposition}

\begin{proof}
  Let \(\mathbf{G_0}\) be the unique \(k\)-split \(k\)-group of type \(D_4\).  By assumption, \(k\) has a real place \(v\).  Let \(l / k\) be any quadratic extension such that \(l \otimes_k k_v \cong \C\).  Let \(a \colon \operatorname{Gal}(k) \rightarrow \operatorname{Sym} \Delta \cong S_3\) be any homomorphism with kernel \(\operatorname{Gal}(l)\) and let \(\mathbf{G_1}\) be the unique quasi-split \(k\)-group of type \(D_4\) whose class in \(H^1(k, \operatorname{Aut} \mathbf{G_0})\) maps to \([a] \in H^1(k, \operatorname{Sym} \Delta)\).  We will construct \(\mathbf{G}\) and \(\mathbf{H}\) as inner \(k\)-forms of \(\mathbf{G_1}\) as follows.

  Pick two finite places \(v_1\) and \(v_2\) of \(k\) that split in \(l\).   Again by \cite{Kammeyer:galois-cohomology}*{Theorem~IV.7.2, p.\,36} and \cite{Baumann-Kammeyer:congruence-rigidity}*{Table~2, p.\,17}, there exists a class \(\alpha \in H^1(k, \operatorname{Ad} \mathbf{G_1})\) that localizes to the unique non-trivial class at \(v\), with Tits diagram
\begin{center}
\begin{tikzpicture}[line width=0.6pt]
  \coordinate (a1) at (0,0);
  \coordinate (a2) at (1.2,0);

  \coordinate (b3) at (1.55,0.22);
  \coordinate (b4) at (1.55,-0.22);
  \coordinate (a3) at (2.35,0.22);
  \coordinate (a4) at (2.35,-0.22);

  \draw (a1) -- (a2);
  \draw (a2) .. controls (1.28,0.14) and (1.40,0.22) .. (b3) -- (a3);
  \draw (a2) .. controls (1.28,-0.14) and (1.40,-0.22) .. (b4) -- (a4);

  \fill (a1) circle[radius=1.6pt];
  \fill (a2) circle[radius=1.6pt];
  \fill (a3) circle[radius=1.6pt];
  \fill (a4) circle[radius=1.6pt];

  \draw (a1) circle[radius=5pt];
\end{tikzpicture} ,
\end{center}
to the unique class with Tits diagram
\begin{center}
\begin{tikzpicture}[line width=0.6pt, baseline=-0.5ex]
  \coordinate (a1) at (0,0);
  \coordinate (a2) at (1.1,0);
  \coordinate (a3) at (2.2,0.65);
  \coordinate (a4) at (2.2,-0.65);

  \draw (a1) -- (a2);
  \draw (a2) -- (a3);
  \draw (a2) -- (a4);

  \foreach \v in {a1,a2,a3,a4}
    \fill (\v) circle[radius=1.6pt];

  \draw (a1) circle[radius=5pt];
  \draw (a2) circle[radius=5pt];
\end{tikzpicture}
\end{center}
based on the Tits diagram of \(\mathbf{G_1}\) at both \(v_1\) and \(v_2\), and to the trivial class at all remaining places of \(k\).  There exists another class \(\beta \in H^1(k, \operatorname{Ad} \mathbf{G_1})\) which localizes to the same class as \(\alpha\) at \(v\) but localizes to the unique class with diagram
\begin{center}
\begin{tikzpicture}[line width=0.6pt, baseline=-0.5ex]
  \coordinate (a1) at (0,0);
  \coordinate (a2) at (1.1,0);
  \coordinate (a3) at (2.2,0.65);
  \coordinate (a4) at (2.2,-0.65);

  \draw (a1) -- (a2);
  \draw (a2) -- (a3);
  \draw (a2) -- (a4);

  \foreach \v in {a1,a2,a3,a4}
    \fill (\v) circle[radius=1.6pt];

  \draw (a2) circle[radius=5pt];
  \draw (a3) circle[radius=5pt];
\end{tikzpicture}
\end{center}
at \(v_1\) and \(v_2\), and to the trivial class at all other places of \(k\).  Let \(\mathbf{G}\) and \(\mathbf{H}\) be the \(k\)-groups corresponding to the images of \(\alpha\) and \(\beta\) in \(H^1(k, \operatorname{Aut} \mathbf{G_1})\), respectively.  By Proposition~\ref{prop:distinguished-orbits}, we have \(\operatorname{rank}_k \mathbf{G} = 1\) whereas \(\operatorname{rank}_k \mathbf{H} = 0\).  Moreover, we have \(\mathbf{G} \cong_{k_v} \mathbf{H}\) for all places \(v\) by construction.  Indeed, at \(v_1\) and \(v_2\), the quasi-split group \(\mathbf{G_1}\) actually splits and the local classes of \(\alpha\) and \(\beta\) lie in the same orbit of \((\Sym \Delta)(k_{v_i}) \cong S_3\).
\end{proof}

\section{Uniform and non-uniform lattices}

The core part for the proofs of our results was given in the previous section, using the machinery of the classification and Galois cohomology of algebraic groups.  To reduce the proofs of our theorems on lattices in semisimple Lie groups to these considerations, it essentially remains to apply Margulis arithmeticity and adelic superrigidity.

\begin{proof}[Proof of Theorem~\ref{thm:same-q-rank}.]
  Replacing \(\Gamma\) and \(\Lambda\) with their images in \(G/Z(G)\), we may assume \(G\) is center-free, hence linear.  By Margulis arithmeticity~\cite{Margulis:discrete-subgroups}*{Theorem~IX.(1.11), p.\,298} and its rank one extensions by Gromov--Schoen~\cite{Gromov-Schoen:harmonic}, and Corlette~\cite{Corlette:superrigidity}, there exist number fields~\(k\) and~\(l\), as well as simply-connected absolutely almost simple linear algebraic groups~\(\mathbf{G}\) and~\(\mathbf{H}\) defined over~\(k\) and~\(l\), respectively, and continuous homomorphisms 
  \begin{equation} \label{eq:alg-lie} \varphi \colon \prod_{v \in S} \mathbf{G}(k_v) \longrightarrow G, \qquad \psi \colon \prod_{w \in T} \mathbf{H}(l_w) \longrightarrow G \end{equation}
  with finite kernel and cokernel, where \(S\) and \(T\) are the subsets of infinite places in \(k\) and \(l\) at which \(\mathbf{G}\) and \(\mathbf{H}\) are isotropic, such that \(\varphi(\mathbf{G}(\mathcal{O}_k))\) is commensurable with \(\Gamma\) and \(\psi(\mathbf{H}(\mathcal{O}_l))\) is commensurable with \(\Lambda\).  Here \(\mathcal{O}_k\) and \(\mathcal{O}_l\) denote the rings of integers in \(k\) and \(l\), and \(\mathbf{G}\) and \(\mathbf{H}\) are embedded into some \(\mathbf{GL}_n\) so that the integral points are defined.

  Since \(G\) is neither isogenous to \(\operatorname{SO}^0(n,1)\), nor to \(\operatorname{SU}(n,1)\), we have that both \(\mathbf{G}\) and \(\mathbf{H}\) are \emph{algebraically superrigid} in the sense of \cite{Kammeyer-Kionke:adelic-superrigidity}*{Definition~3.1 and below}.  Since  \(\widehat{\Gamma}\) and \(\widehat{\Lambda}\) are commensurable, \emph{adelic superrigidity}~\cite{Kammeyer-Kionke:adelic-superrigidity}*{Theorem~3.2} applies and we obtain a bijection \(v \mapsto v'\) of the finite places of \(l\) and \(k\), isomorphisms \(\eta_v \colon l_v \rightarrow k_{v'}\), and isomorphisms \(\mathbf{G} \cong_{\eta_v} \mathbf{H}\).  In particular, the fields \(k\) and \(l\) are arithmetically equivalent, hence have the same signature~\cite{Klingen:similarities}*{Theorem~III.1.4.h, p.\,79}.  Since \(\varphi\) and \(\psi\) have finite kernel and cokernel, and since \(\mathbf{G}\) and \(\mathbf{H}\) are simply-connected, we can extend the bijection \(v \mapsto v'\) and the corresponding isomorphisms of local fields and groups to the infinite places of \(k\) and \(l\) where \(\mathbf{G}\) and \(\mathbf{H}\) are isotropic.  Finally, since \(k\) and \(l\) have the same signature, they have the same number of real places where \(\mathbf{G}\) and \(\mathbf{H}\) are anisotropic.  Since each simple Cartan--Killing type has a unique simply-connected anisotropic real form, it follows that we can extend the bijection and corresponding isomorphism to all places of \(l\) and \(k\).  As \(G\) has absolute type different from \(D_{2n}\), so do \(\mathbf{G}\) and \(\mathbf{H}\).  Thus Corollary~\ref{cor:rank} completes the proof.
\end{proof}  

\begin{proof}[Proof of Theorem~\ref{thm:different-q-rank}.]
  The first part of the theorem is immediate from the special case \(k = \Q\) of Proposition~\ref{prop:differing-ranks}.  In fact, the \(\Q\)-groups \(\mathbf{G}\) and \(\mathbf{H}\) defined in the proof of Proposition~\ref{prop:differing-ranks} are isomorphic to \(G = \operatorname{Spin}(2n,2n)\) at the real place of \(\Q\).  So any two arithmetic subgroups \(\Gamma \le \mathbf{G}\) and \(\Lambda \le \mathbf{H}\) satisfy all claims of the theorem because \(\mathbf{G}\) and \(\mathbf{H}\) have higher real rank and are \(\Q\)-isotropic by Proposition~\ref{prop:distinguished-orbits}, hence have CSP by the work of Prasad--Raghunathan~\cite{Platonov-Rapinchuk:algebraic-groups}*{Theorem~9.17, p.\,558 and p.\,568}.  Together with strong approximation, this shows that \(\widehat{\Gamma}\) and \(\widehat{\Lambda}\) are commensurable with open subgroups of \(\mathbf{G}(\mathbb{A}^f_k)\) and \(\mathbf{H}(\mathbb{A}^f_k)\), respectively, where \(\mathbb{A}^f_k\) denotes the \(k\)-algebra of finite adeles of~\(k\).  The latter two topological groups are isomorphic because \(\mathbf{G}\) and \(\mathbf{H}\) localize to the same isomorphism types at finite places.

  For the second part of the theorem, suppose \(\Gamma, \Lambda \le G\) are irreducible lattices in a type \(D_{2n}\)-group as stated with commensurable profinite completions and \(\operatorname{rank}_\Q(\Gamma) > \operatorname{rank}_\Q(\Lambda)\).  Arguing as in the previous proof, there are a \(k\)-group \(\mathbf{G}\) and an \(l\)-group \(\mathbf{H}\) whose arithmetic subgroups are commensurable with \(\Gamma\) and \(\Lambda\), respectively, such that \(\mathbf{G}\) and \(\mathbf{H}\) satisfy the assumptions of Proposition~\ref{prop:only-n-and-n-1}.  We conclude \(\operatorname{rank}_\Q (\Gamma) = n\) and \(\operatorname{rank}_\Q (\Lambda) = n-1\).
\end{proof}

\begin{proof}[Proof of Theorem~\ref{thm:triality}.]
The theorem follows from Proposition~\ref{prop:triality} by choosing any number field \(k\) that is neither \(\Q\) nor totally imaginary.  Indeed, in that case \(k\) has at least two infinite places and thus the groups \(\mathbf{G}\) and \(\mathbf{H}\) constructed in the proof of Proposition~\ref{prop:triality} have higher rank.  Again, CSP is known for \(\mathbf{G}\) as it is \(k\)-isotropic, and it holds for \(\mathbf{H}\) by assumption.  So arithmetic subgroups of \(\mathbf{G}\) and \(\mathbf{H}\) define lattices as required.
\end{proof}

The statement below Theorem~\ref{thm:uniform-non-uniform} is also clear because we can choose a real quadratic field for \(k\) and repeat the construction with a quadratic extension \(l/k\) in which both real places of \(k\) become complex.  We can then prescribe that the rank one form of type \({}^2 D_4\) occurs at both real places of \(k\) for \(\mathbf{G}\) and \(\mathbf{H}\).

\begin{proof}[Proof of Theorem~\ref{thm:different-lie-groups}.]
  Let \(\Gamma \le G\) be a uniform lattice.  Since \(G\) is connected and has finite center, we may replace \(\Gamma\) with a finite index subgroup, if need be, to assume that \(\Gamma\) is a lattice in \(\mathbf{G}(\mathbb{K})\) where \(\mathbb{K} = \R \text{ or } \C\) and where \(\mathbf{G}\) is a connected absolutely almost simple \(\mathbb{K}\)-group.  Since \(G\) is neither compact nor of type \(F_{4(-20)}\), we have \(\operatorname{rank}_{\mathbb{K}} \mathbf{G} \ge 2\).  By Margulis arithmeticity~\cite{Margulis:discrete-subgroups}*{Theorem IX.1.11, p.\,298 and (\(\ast \ast\)) on p.\,293}, there exists a densely embedded number field \(k \subset \mathbb{K}\) and a simply connected absolutely almost simple \(k\)-group \(\mathbf{H}\) with the following properties.
  \begin{enumerate}[label=(\roman*)]
  \item \label{item:lie-algebras} The Lie groups \(\mathbf{H}(\mathbb{K})\) and \(\mathbf{G}(\mathbb{K})\) have \(\mathbb{K}\)-isomorphic Lie algebras.  
  \item The infinite places of \(k\) other than the given one are real.
  \item The group \(\mathbf{H}\) is anisotropic at these other real places.
  \item \label{item:lattices-commensurable} The lattice \(\Gamma\) is abstractly commensurable with every arithmetic subgroup of \(\mathbf{H}\).
  \end{enumerate}
  By~\ref{item:lie-algebras}, the group \(\mathbf{H}\) has type \(E_8\), \(F_4\), or \(G_2\), so the Dynkin diagram has no symmetries and \(\mathbf{H}\) is both simply connected and has trivial center.  Thus \(\mathbf{H} = \operatorname{Aut} \mathbf{H}\).  A theorem of Kneser~\cite{Kneser:galois} shows that \(H^1(K, \mathbf{H})\) is trivial for every \(\mathfrak{p}\)-adic field \(K\).   So \(\mathbf{H}\) as well as every \(k\)-form of \(\mathbf{H}\) splits at all finite places of \(k\).  Now let \(\mathbf{H_0}\) be the global \(k\)-split form of \(\mathbf{H}\) and set \(H = \prod_{v \mid \infty} \mathbf{H_0}(k_v)\).  Then by the Borel--Harish-Chandra theorem, we can choose any arithmetic subgroup \(\Lambda \le \mathbf{H_0}(k)\) to obtain an irreducible lattice in \(H\) and this lattice is non-uniform because \(\operatorname{rank}_k \mathbf{H_0} = 8, 4, \text{or } 2\) depending on the type of \(G\).  Note that not only the \(k\)-split group \(\mathbf{H_0}\) but also the \(k\)-anisotropic group \(\mathbf{H}\) is known to have a finite congruence kernel~\cite{Platonov-Rapinchuk:algebraic-groups}*{Theorems~9.23 and~9.24, p.\,568-569} since \(\operatorname{rank}_{\mathbb{K}} \mathbf{H} \ge 2\) by~\ref{item:lie-algebras}.  Together with~\ref{item:lattices-commensurable}, CSP gives that the profinite completion \(\widehat{\Gamma}\) is commensurable with an open compact subgroup of \(\mathbf{H}(\mathbb{A}^f_k)\) while \(\widehat{\Lambda}\) is commensurable with an open compact subgroup of \(\mathbf{H_0}(\mathbb{A}^f_k)\) where \(\mathbb{A}^f_k\) denotes the ring of finite adeles of \(k\).  But since both \(\mathbf{H}\) and \(\mathbf{H_0}\) split at all finite places of \(k\), we have \(\mathbf{H}(\mathbb{A}^f_k) \cong \mathbf{H_0}(\mathbb{A}^f_k)\), hence \(\widehat{\Gamma}\) is commensurable with \(\widehat{\Lambda}\).
\end{proof}

  Finally, to prove the statement below Theorem~\ref{thm:different-q-rank}, we argue similarly as in the proof of Proposition~\ref{prop:only-n-and-n-1}.  Indeed, if two groups of type \(D_{2n}\) have the same local isomorphism types but differing global ranks, they must both be inner forms and the Tits diagrams arising at real places have at least one of the two end nodes of the short two legs of the Dynkin diagram distinguished.  This leaves \(\operatorname{SO}^0(2n,2n)\) and \(\operatorname{SO}^*(4n)\) as the only groups that may occur at real places.  If only these groups occur as simple real factors of \(G\), we can argue similarly to the proof of Proposition~\ref{prop:differing-ranks} to construct two \(k\)-groups with differing \(k\)-rank whose arithmetic subgroups define profinitely commensurable lattices in \(G\).

\end{document}